\documentclass[leqno,11.5pt]{article}
\usepackage{latexsym}
\usepackage{amsmath}
\usepackage{amssymb}
\usepackage{enumerate}
\usepackage{graphicx}
\usepackage[all]{xy}

\usepackage{theorem}
\newtheorem{theorem}{Theorem}[section]

\newtheorem{lemma}[theorem]{Lemma}
\newtheorem{corollary}[theorem]{Corollary}

\theorembodyfont{\rmfamily}
\newtheorem{proof}{\textmd{\textit{Proof.}}}

\newtheorem{remark}[theorem]{Remark}

\newtheorem{acknowledgement}{\textmd{\textit{Acknowledgements.}}}

\makeatletter

\@addtoreset{equation}{section}
\makeatother

\newcommand{\qedd}{\hfill \Box}

\newcommand{\lra}{\longrightarrow}

\newcommand{\wt}{\widetilde}

\newcommand{\B}{\ensuremath{\mathbb{B}}}

\newcommand{\R}{\ensuremath{\mathbb{R}}}
\newcommand{\bH}{\ensuremath{\mathbb{H}}}
\newcommand{\Sph}{\ensuremath{\mathbb{S}}}

\def\vol{\mathop{\mathrm{vol}}\nolimits}

\title{
\Large{
On sufficient conditions to extend 
Huber's finite\\ connectivity theorem to higher dimensions}
\\
{\small {\it Dedicated to Professor K. Shiohama on his eightieth birthday}}
}

\author{Kei KONDO\footnote{Supported by the JSPS KAKENHI Grant Numbers 17K05220, and partially 16K05133, 18K03280.} \   
and Yusuke SHINODA
}
\date{\today}
\begin{document}
\maketitle

\begin{abstract}
Let $M$ be a connected complete 
noncompact $n$-dimensional Riemannian manifold 
with a base point $p \in M$ whose radial sectional curvature 
at $p$ is bounded from below by that of a noncompact surface 
of revolution which admits a finite total curvature where 
$n \ge 2$. 
Note here that our radial curvatures can change signs wildly. 
We then show that 
$\lim_{t\to\infty}\vol B_t(p) / t^n$ exists 
where $\vol B_t(p)$ denotes the volume of 
the open metric ball $B_t(p)$ with center $p$ and radius $t$. 
Moreover we show that in addition if the limit above 
is positive, then $M$ has finite topological 
type and there is therefore a finitely upper bound on the 
number of ends of $M$. 
\end{abstract}

\section{Introduction}\label{sec1}%%%%%%%%%%%%%%%%%%%%%%%%%%%%%%%%%%%%%%%%%%%%%%%%%%%%%%%%%%%%%%%%%%%%%%%%%%%%%%%%%%%%%%%%%%%%%%%%%%%%%%%%%%%%%%%%%%%%%%%%%%%%%%%%%%%%%%%%%%%%%%%%%%%%%%%%%%%%%%%%%%%%%%%%%%%%%%%%

It is a great delight to try generalizing classic results of  
the relationships between a total curvature and topology 
of a connected complete noncompact $2$-dimensional 
Riemannian manifold $S$ to higher dimensions. The reason for our joy is largely due to a well-known theorem of 
Cohn-Vossen \cite{CV1} in 1935 which states that 
if $S$ is {\em finitely connected} and admits a total 
curvature $c(S)$ as an extended real number, 
then $c(S)\le 2\pi \chi (S)$ holds 
where $\chi (S)$ denotes the 
Euler characteristic of $S$, and hence $c(S)$ 
is not a topological invariant anymore. 
Here $S$ is finitely connected 
if there is a compact $2$-dimensional 
Riemannian manifold $V$ and finite numbers of points 
$p_1, p_2, \ldots, p_k \in V$ ($k\ge1$) 
such that $S$ is homeomorphic to 
$V \setminus \{p_1, p_2, \ldots, p_k\}$. 
Besides, Cohn-Vossen showed, applying the theorem above, that if $S$ has nonnegative 
Gaussian curvature everywhere, 
then $S$ is either diffeomorphic to a plane or else 
isometric to a flat cylinder or a flat open M\"obius strip. As one of extensions of Cohn-Vossen's 
theorem above to higher dimensions, Cheeger and Gromoll 
\cite{CG2} proved the soul theorem in 1972, 
which states that 
for every complete noncompact Riemannian 
manifold $X$ with nonnegative sectional curvature everywhere there is a compact totally geodesic submanifold, called the soul, of $X$ such that $X$ is diffeomorphic to the normal bundle over the soul.\par 
Huber \cite{Hub} showed in 1957 that if $S$ admits 
a finite total curvature, which implies 
$c(S)>-\infty$ by Cohn-Vossen's 
theorem above, then $S$ is finitely connected. 
As a weak version of the soul theorem, 
Gromov \cite{G} showed, applying the Grove-Shiohama 
theory \cite{GS} of critical points of distance functions, 
in 1981 that the $X$ above {\em has finite topological type}, i.e., 
$X$ is homeomorphic to the interior of a compact manifold with boundary and consequently, 
we can also see his result as an extension of Huber's 
theorem. What should be noted here is the 
fact that $X$ must employ the Euclidean space 
of dimension $2$, denoted by $\R^2$, 
as a reference surface in comparison theorems such as 
the Toponogov comparison theorem. The total 
curvature of $\R^2$ is thus $0$, but 
Huber's theorem also 
holds for all $S$ whose total curvatures take 
a finite value in $(-\infty, 2\pi]$ other 
than $0$.\par  
It is a meaningful observation on the 
geometry of total curvatures that we 
restrict $c(S)$ from being bounded from below by a constant, because $c(S)$ is not a topological invariant. 
For example the following theorem of Shiohama is remarkable about the restriction. 

\begin{theorem}{\rm (\cite{Sh})}\label{Sh}
Let $S$ be oriented, finitely connected, and have 
one end. If $c(S) > (2 \chi (S) - 1) \pi$, 
then all Busemann functions on $S$ are exhaustions. 
In particular, if $c(S) > \pi$, 
then $S$ is homeomorphic to $\R^2$ 
and also all Busemann functions on $S$ 
are exhaustions. 
\end{theorem} 
Is it appropriate for us to extend Theorem \ref{Sh} 
to complete noncompact Riemannian manifolds 
of higher dimensions which have all curvatures 
everywhere bounded from below by $0$, 
or $-1$? No, it is not, 
for total curvatures of reference surfaces, 
$\R^2$, or the hyperbolic plane $\bH^2(-1)$,  
in comparison theorems for such manifolds are 
$0$ corresponding to $\R^2$, or 
$-\infty$ corresponding to $\bH^2(-1)$, which 
are less than $\pi$. It is therefore natural that 
we desire {\em radial curvature geometry} to extend 
Huber's finite connectivity theorem, 
Theorem \ref{Sh}, and other classic results that 
we can find in \cite{SST} to higher dimensions, 
because Gaussian curvatures of reference surfaces 
in the radial curvature geometry can wildly change their 
signs so that we can restrict total curvatures of the 
surfaces from being bounded from below by a constant.\par  
We will introduce the radial curvature geometry. 
Let $\wt{M}^{n}$ be a simply connected complete 
$n$-dimensional Riemannian manifold with 
a base point $\tilde{p} \in \wt{M}^{n}$ 
where $n\ge2$, and 
$\tilde{d}$ the distance function of $\wt{M}^n$. 
Set $\ell:= \sup_{\tilde{x}\in \wt{M}^n}\tilde{d}(\tilde{p}, \tilde{x}) \le \infty$ and 
$
\Sph^{n - 1}_{\tilde{p}} 
:=
\{ 
v \in T_{\tilde{p}} \wt{M}^{n} \ | \ \| v \| = 1
\}
$ 
where $T_{\tilde{p}}\wt{M}^n$ denotes the tangent space at $\tilde{p} \in \wt{M}^n$. We then call the pair $(\wt{M}^{n}, \tilde{p})$ the 
{\em general $n$-dimensional comparison space} 
if its Riemannian metric $d\tilde{s}^2$ is expressed 
in terms of geodesic polar coordinates 
around $\tilde{p}$ as 
\begin{equation}\label{polar}
d\tilde{s}^2 = dt^2 + f(t)^2 d\tilde{s}^2_{\Sph^{n-1}_{\tilde{p}}}(\theta)
\end{equation}
for all 
$(t,\theta) \in (0,\ell) \times \Sph_{\tilde{p}}^{n -1}$. 
In Eq.\,\eqref{polar}, 
$f : (0, \ell) \lra \R$ is the warping function 
of $\wt{M}^{n}$, which is, by definition, 
a positive smooth function satisfying the 
Jacobi equation 
\begin{equation}\label{Jacobi}
f''(t) + G (\tilde{\gamma}(t)) f(t) = 0
\end{equation} 
with initial conditions $f(0) = 0$ and $f'(0) = 1$ 
where $G$ denotes the sectional curvatures containing a radial direction of $\wt{M}^{n}$ and $\tilde{\gamma}$ denotes any meridian emanating from 
$\tilde{p} = \tilde{\gamma} (0)$, and $d\tilde{s}_{\Sph^{n-1}_{\tilde{p}}}$ is the Riemannian metric on $
\Sph^{n - 1}_{\tilde{p}}$. 
The function 
$\wt{K}:= G \circ \tilde{\gamma} : [0,\ell) \lra \R$ 
is called 
the {\em radial curvature function} of $\wt{M}^{n}$. In the case of 
$n = 2$ we simply call $(\wt{M}^2, \tilde{p})$ 
the {\em surface of revolution}, and then 
Eq.\,\eqref{polar} is also expressed simply 
as 
\begin{equation}\label{polar_surface}
d\tilde{s}^2 = dt^2 + f(t)^2 d\theta^2
\end{equation} 
for all $(t,\theta) \in (0,\ell) \times \Sph_{\tilde{p}}^1$. 
The total curvature $c(\wt{M}^2)$ of $\wt{M}^2$ is thus 
given by 
\begin{equation}\label{TC_surface}
c(\wt{M}^2)
= \int_{\wt{M}^2} \wt{K}_+\,d\wt{M}^2 
+ \int_{\wt{M}^2} \wt{K}_-\,d\wt{M}^2
\end{equation}
if $\int_{\wt{M}^2} \wt{K}_+\,d\wt{M}^2 < \infty$ or 
$\int_{\wt{M}^2} \wt{K}_-\,d\wt{M}^2 > -\infty$ where 
$\wt{K}_+:= \max\{\wt{K}, 0\}$, 
$\wt{K}_-:= \min\{\wt{K}, 0\}$, and 
$d\wt{M}^2$ is the area element of $\wt{M}^2$, 
i.e., $d\wt{M}^2= f(t) dt d\theta$. In particular 
$\chi (\wt{M}^2)=1$, 
for $\wt{M}^2$ is homeomorphic to $\R^2$. 

\begin{remark}
Katz and the first author 
\cite{KK} classified connected complete Riemannian manifolds with symmetric radial curvature, and hence general comparison spaces were classified. Moreover 
in their energetic study Mao and his collaborators employ 
such spherically symmetric manifolds to develop geometric analysis and inequalities to a wider class 
of metrics, see \cite{FMS}, \cite{Mao1}, \cite{Mao2}, 
and so on. 
\end{remark}
We now say that 
a connected complete $n$-dimensional Riemannian manifold $M$ with a base point $p \in M$ {\em has 
radial curvature at $p$ bounded from below by 
that of a surface of revolution $(\wt{M}^2, \tilde{p})$} if, 
along every unit speed minimal geodesic $\gamma: [0,a) \lra M$ 
emanating from $p = \gamma (0)$, 
\[
K_M(\gamma'(t), v) \ge \wt{K} (t)
\]
holds for all $t \in [0, a)$ and all 
$v \in T_{\gamma (t)}M$ with $v \perp \gamma'(t)$ 
where $K_M(\gamma'(t), v)$ denotes 
the sectional curvature of $M$ restricted to 
the $2$-dimensional linear space 
in $T_{\gamma (t)}M$ spanned by $\gamma'(t)$ 
and $v$, and 
$\wt{K}$ denotes the radial curvature 
function of $\wt{M}^2$. 

\medskip

Our purpose of this article is to extend Huber's 
finite connectivity theorem to higher dimensions in radial curvature geometry, and our main theorem is stated as follows:

\begin{theorem}\label{2019_10_28_maintheorem}
Let $M$ be a connected complete noncompact 
$n$-dimensional Riemannian manifold with 
a base point $p\in M$ whose radial curvature 
at $p$ is bounded from below by that 
of a noncompact surface of revolution 
$(\wt{M}^2, \tilde{p})$ with its metric 
\eqref{polar_surface} where $n\ge 2$. 
Assume $\int_{\wt{M}^2} \wt{K}_-\,d\wt{M}^2 > -\infty$. 
We then have that 
\begin{enumerate}[{\rm (1)}]
\item
$\lim_{t\to\infty}\vol B_t(p) / t^n$ exists where $\vol B_t(p)$ denotes the volume of 
the open metric ball $B_t(p)$ with center $p$ and radius $t$; 
\item 
in addition, if the limit is nonzero, then 
$M$ has finite topological type and 
the number of ends of $M$ is less than or equal 
to $2 \{\lim_{t\to\infty}m'(t) \}^{n-1}$. Here 
$m: (0, \infty)\lra \R$ is a positive function of class $C^r$, 
$r\ge 2$, satisfying the 
Jacobi equation $m''(t) + \wt{K}_- (t) m(t) = 0$ 
with $m(0) = 0$ and $m'(0) = 1$. 
\end{enumerate}
\end{theorem}

\begin{remark}\label{2020_04_11_Rem1.4}
We here give remarks on Theorem 
\ref{2019_10_28_maintheorem} and 
related results to the theorem. 
\begin{enumerate}
\item From Eq.\,\eqref{TC_surface},  
$\int_{\wt{M}^2} \wt{K}_-\,d\wt{M}^2 > -\infty$ if and only if 
$c(\wt{M}^2) > -\infty$. 
\item The assumption 
$\lim_{t\to\infty}\vol B_t(p) / t^n \not = 0$ 
guarantees 
$2 \{\lim_{t\to\infty}m'(t) \}^{n-1}< \infty$. Indeed, 
from the assumption,  
$c(\wt{M}^2) \in (-\infty, 2\pi)$ holds 
(see Corollary \ref{2019_11_25_Cor3.3}), 
and hence the argument in 
the proof of \cite[Lemma 2.2]{KT5} shows 
$\int_0^\infty t \wt{K}_-(t)\,dt >-\infty$, which 
implies $\lim_{t\to\infty}m'(t) < \infty$ 
(see the proof of \cite[Theorem 5.3]{KT1}). Note that 
$m' (t)\ge 1$ on $[0,\infty)$, for $m''(t)= -\wt{K}_-(t)m(t)\ge 0$ and $m'(0)=1$. 
\item As was noted above, our 
radial curvatures can change signs wildly. For example 
there is a noncompact surface of revolution 
admitting a finite total curvature whose 
radial curvature function $\wt{K}$ is not bounded, which 
means $\liminf_{t \to \infty} \wt{K}(t) = - \infty$, or 
$\limsup_{t \to \infty} \wt{K}(t)= \infty$,  
see \cite[Theorem 1.4]{KT5}.
\item The first author and Tanaka \cite{KT1} extended 
Huber's finite connectivity theorem to higher dimensions as follows: Let $M$ be a connected complete noncompact 
$n$-dimensional Riemannian manifold with a base point $p\in M$ whose 
radial curvature at $p$ is bounded from below 
by that of a noncompact surface of revolution 
$(\wt{M}^2, \tilde{p})$ with its metric 
\eqref{polar_surface} where $n \ge2$. If 
\begin{enumerate}[{\rm (a)}]
\item
$\wt{M}^2$ admits a finite total curvature, and if 
\item
$\wt{M}^2$ has no pair of cut points in the set 
$\wt{V}(\delta_{0}) := \{ \tilde{x} \in \wt{M}^2 \, | \, 0 < \theta(\tilde{x}) < \delta_{0} \}$ for some $\delta_{0} \in (0, \pi]$, 
\end{enumerate}
then $M$ has finite topological type.
\item The two assumptions (a) and (b) above can be 
replaced with one assumption 
$c(\wt{M}^2) \in (-\infty, 2\pi)$, 
see \cite[Theorem 1.3]{KT5} of the first author 
and Tanaka.
\item Abresch and Gromoll \cite{AG} also 
showed the finiteness of topological type of 
a complete noncompact $n$-dimensional Riemannian manifold $X$ having not only nonnegative Ricci curvature 
outside the open distance $t_{0}$-ball around $p \in X$ for some constant $t_{0} > 0$, but also sectional curvature everywhere bounded from below by a negative 
constant, and moreover admitting diameter growth of small order $o(t^{1/ n})$. 
Note that we find that the diameter growth is too restrictive 
from the radial curvature geometry point of view, 
see \cite[example 1.1]{KT1}.
\item Abresch \cite{A} obtained an upper bound on the 
number of ends of an asymptotically nonnegatively curved 
manifold, which is a manifold whose radial 
curvature at a base point is bounded from below by the radial curvature function, denoted by $\wt{F}$, 
satisfying $\wt{F} <0$, $\wt{F}' \ge 0$, 
and $\int_0^{\infty}-t\wt{F}(t) \, dt <\infty$. 
Note that these conditions implies that the 
reference surface of such a manifold is an Hadamard 
one admitting a finite total curvature. 
So we can say that our class of metrics in Theorem \ref{2019_10_28_maintheorem} 
are wider than that of 
asymptotically nonnegatively curved metrics. 
\item Hebda and Ikeda showed, by using their Toponogov comparison theorem \cite{HI1}, that if a noncompact surface 
of revolution $(\wt{M}^2, \tilde{p})$ with its metric 
\eqref{polar_surface} has weaker radial attraction than a connected complete noncompact Riemannian manifold $M$ 
with a base point $p\in M$ and minimal geodesics in $M$ have no bad encounters with cut loci in $\wt{M}^2$ in their sense, 
and if 
$\liminf_{t\to\infty}f(t)/t< 2/\pi$, then $M$ has at most one end (\cite[Proposition 7.3]{HI2}).  
\end{enumerate}
\end{remark}

\begin{remark}
Theorem \ref{Sh} had been extended to higher dimensions in radial curvature geometry by the first author and Tanaka \cite{KT2} as follows: Let $M$ be a connected complete noncompact 
$n$-dimensional Riemannian manifold with a base 
point $p\in M$ whose radial curvature at $p$ is 
bounded from below by that of a noncompact von 
Mangoldt surface of revolution $(\wt{M}^2, \tilde{p})$. 
If $c(\wt{M}^2)>\pi$, then all Busemann functions 
of $M$ are exhaustions. Here a von Mangoldt surface 
of revolution is, by definition, a surface of revolution whose 
radial curvature 
function is nonincreasing on $[0,\infty)$. They further 
extended this result to an $M$ which is 
not less curved than a more general surface of revolution, see \cite{KT4}.\end{remark}

\begin{remark}
This article is a part of the master thesis of the second author, Yamaguchi University. 
\end{remark}

\begin{acknowledgement}
We thank the referee for a careful reading of the manuscript, valuable suggestions, and helpful comments on the manuscript, which have improved the presentation of this article.
\end{acknowledgement}

\section{Proof of Theorem \ref{2019_10_28_maintheorem}}\label{sec2}%%%%%%%%%%%%%%%%%%%%%%%%%%%%%%%%%%%%%%%%%%%%%%%%%%%%%%%%%%%%%%%%%%%%%%%%%%%%%%%%%%%%%%%%%%%%%%%%%%%%%%%%%%%%%%%%%%%%%%%%%%%%%%%%%%%%%%%%%%%%%%%%%%%%%%%%%%%%%%%%%%%%%%%%%%%%%%%%

Throughout this section 
let $M$ be a connected complete noncompact 
$n$-dimensional Riemannian manifold with a base point $p$ whose 
radial curvature at $p$ is 
bounded from below by the radial curvature 
function $\wt{K}$ of a noncompact surface of revolution $(\wt{M}^2, \tilde{p})$, and let $d\tilde{s}^2$ denote 
the Riemannian metric of $\wt{M}^2$ given by 
Eq.\,\eqref{polar_surface} with the warping function 
$f:(0,\infty)\lra \R$ satisfying Eq.\,\eqref{Jacobi}. 
In these settings we assume 
$\int_{\wt{M}^2} \wt{K}_-\,d\wt{M}^2 > -\infty$ where 
$\wt{K}_- (t)= \min \{\wt{K} (t), 0\}$ on $[0,\infty)$. 
The total curvature $c(\wt{M}^2)$ of $\wt{M}^2$ is then finite, see 
the first one of Remark \ref{2020_04_11_Rem1.4}. 
In particular 
Cohn-Vossen's theorem gives 
$c(\wt{M}^2)\in (-\infty, 2\pi]$, 
for $\chi (\wt{M}^2)=1$.\par 
Let $(\wt{M}^n, \tilde{o})$ be 
the noncompact general
$n$-dimensional comparison space 
whose warping function is $f$ of $\wt{M}^2$ 
where $\tilde{o} \in \wt{M}^n$ denotes the base point 
of it. Note that the radial curvature of $M$ at $p$ is also 
bounded from below by that 
of $(\wt{M}^n, \tilde{o})$. Let 
$B_t (\tilde{o}) \subset \wt{M}^n$ be 
the metric open ball with center $\tilde{o}$ 
and radius $t$, and $\omega_{n-1}$ 
the volume of 
$
\Sph^{n-1}_{\tilde{o}}
=
\{v \in T_{\tilde{o}} \wt{M}^n\,|\, \|v\|=1 
\}$. 

\begin{lemma}\label{2019_10_28_lem2.1}
$\lim_{t\to\infty}\vol B_t(p) / t^n$ exists. 
\end{lemma}

\begin{proof}
Since $\vol B_t (\tilde{o}) = \omega_{n-1} \int_0^t f(r)^{n-1}dr$, we have 
\begin{equation}\label{2019_10_28_lem2.1_1}
\lim_{t\to\infty}\frac{\vol B_t(p)}{t^n} 
= 
\omega_{n-1} \cdot \lim_{t\to\infty}
\frac{\vol B_t(p)}{\vol B_t (\tilde{o})} 
\cdot 
\frac{\int_0^t f(r)^{n-1}dr}{t^n}.
\end{equation}
The Bishop volume comparison theorem \cite{BC} 
and the Bishop--Gromov one 
in radial curvature geometry \cite{Mao1} show 
that 
$\lim_{t\to\infty}
\vol B_t(p) / \vol B_t (\tilde{o})
$ exists in $[0,1]$. 
Moreover since $f$ is smooth and 
$\lim_{t\to\infty}t^{n-1}= \infty$, 
it follows from l'H\^opital's theorem (cf.\,\cite[Lemma 5.2.1]{SST}) 
and the isoperimetric inequality (\cite[Theorem 5.2.1 (5.2.2)]{SST}) that 
\begin{align}\label{2019_10_28_lem2.1_2}
\lim_{t\to\infty} \frac{\int_0^t f(r)^{n-1}dr}{t^n} 
&= \lim_{t\to\infty} \frac{f(t)^{n-1}}{n \cdot t^{n-1}} 
= \frac{1}{n \cdot (2\pi)^{n-1}}\lim_{t\to\infty} 
\Big\{
\frac{2 \pi f(t)}{t} 
\Big\}^{n-1}\\[1mm]
&=\frac{1}{n \cdot (2\pi)^{n-1}} 
\{
2\pi -c(\wt{M}^2)
\}^{n-1} 
=\frac{1}{n} \Big\{
1 -\frac{c(\wt{M}^2)}{2\pi}
\Big\}^{n-1}.\notag
\end{align}
We therefore see, by Eqs.\,\eqref{2019_10_28_lem2.1_1} and \eqref{2019_10_28_lem2.1_2}, that 
$\lim_{t\to\infty}\vol B_t(p) / t^n \in [0,\infty)$, 
for $c(\wt{M}^2) \in (-\infty, 2\pi]$.$\qedd$ 
\end{proof}

In addition we now assume 
$\lim_{t\to\infty} \vol B_t(p) / t^n \not =0$.

\begin{lemma}\label{2019_11_01_lem2.2}
$M$ has finite topological type. 
\end{lemma}

\begin{proof}
By the additional assumption 
there are two positive constants 
$\alpha_1, \alpha_2$ ($\alpha_1 < \alpha_2$) such that 
$\lim_{t\to\infty}\vol B_t(p) / t^n \in [\alpha_1,\alpha_2]$. From this,  
Eqs.\,\eqref{2019_10_28_lem2.1_1} and \eqref{2019_10_28_lem2.1_2} show 
\begin{equation}\label{2019_11_01_lem2.2_1}
\alpha_1\le 
\frac{\omega_{n-1} }{n} \cdot
\Big\{
1 -\frac{c(\wt{M}^2)}{2\pi}
\Big\}^{n-1}
\cdot \lim_{t\to\infty}
\frac{\vol B_t(p)}{\vol B_t (\tilde{o})} 
\le \alpha_2.
\end{equation}
Since $\alpha_1 >0$, $\lim_{t\to\infty}
\vol B_t(p) / \vol B_t (\tilde{o})$ is positive, 
and hence, by Eq.\,\eqref{2019_11_01_lem2.2_1}, 
we obtain two positive constants 
$\beta_1 (n), \beta_2 (n)$ given by 
\[
\beta_i (n) 
:= 
\frac{n \cdot \alpha_i}{
\omega_{n-1} \cdot 
\displaystyle{
\lim_{t\to\infty}
\frac{\vol B_t(p)}{\vol B_t (\tilde{o})} 
}
}
\]
for each $i=1,2$ such that 
\begin{equation}\label{2019_11_01_lem2.2_2}
\beta_1 (n) 
\le \Big\{
1 -\frac{c(\wt{M}^2)}{2\pi}
\Big\}^{n-1}
\le \beta_2 (n).
\end{equation}
Eq.\,\eqref{2019_11_01_lem2.2_2} thus implies that 
$c(\wt{M}^2)$ is the finite total curvature less than $2\pi$. 
It therefore follows from \cite[Theorem 1.3]{KT5} that 
$M$ has finite topological type. 
$\qedd$ 
\end{proof}

From the argument in the proof of Lemma \ref{2019_11_01_lem2.2} we have the following 
corollary. 

\begin{corollary}\label{2019_11_25_Cor3.3}
$c(\wt{M}^2) \in (-\infty, 2\pi)$. 
\end{corollary}

We finally estimate an upper bound on the number of ends of $M$: An end of $M$ is, by definition, an equivalence class of cofinal curves on $M$. Here two curves $\alpha, \beta :[0, \infty)\lra M$ are said to be {\em cofinal} if for any compact set $D\subset M$ there is a number $t_0 >0$ such that if $t_1, t_2 \ge t_0$, then $\alpha (t_1)$ and $\beta(t_2)$ are contained in the same connected component of $M \setminus D$. Remark that for any point $x \in M$ 
every end of $M$ contains a ray emanating from $x$. Let 
${\rm Ends}(M)$ denote the set of all ends of $M$. 

\begin{lemma}\label{2019_11_09_lem3.3}
$\#{\rm Ends}(M) \le 2 \{\lim_{t\to\infty}m'(t) \}^{n-1}$ where 
$m: (0, \infty)\lra \R$ is a positive $C^r$-function satisfying the 
Jacobi equation $m''(t) + \wt{K}_- (t) m(t) = 0$ 
with $m(0) = 0$ and $m'(0) = 1$ ($r\ge 2$). 
\end{lemma}

\begin{proof}Since $M$ is complete noncompact, $M$ has at least one end, and hence we can assume, for our purpose, 
$\#{\rm Ends}(M)\ge 2$. 
For each ${\bf e}\in {\rm Ends}(M)$ choose 
$v_{\bf e} \in \Sph_p^{n-1}:=\{v \in T_pM\,|\,\|v\|=1\}$ 
which is the initial velocity vector 
of the ray $\gamma_{v_{\bf e}}:[0,\infty)\lra M$ 
emanating from $\gamma_{v_{\bf e}}(0)=p$ 
such that $\gamma_{v_{\bf e}} \in {\bf e}$. 
There is then a number $\lambda > 0$ 
such that for any distinct ends ${\bf e}, {\bf e}'\in {\rm Ends}(M)$
the open balls 
$\B_\lambda (v_{\bf e}), \B_\lambda (v_{{\bf e}'})
\subset \Sph_p^{n-1}$ with centers $v_{\bf e}, v_{{\bf e}'}$ and the same 
radius $\lambda$, respectively, are mutually disjoint. In order to get the $\lambda$ we need the Toponogov comparison 
theorem (TCT) in radial curvature geometry. 
All geodesic triangles of TCT in such a geometry 
must have the base point as one of their vertices. Moreover if for some $\delta \in (0,\pi]$ 
the sector $\wt{V}(\delta):=\{x \in \wt{M}^2\,|\,0<\theta(x) < \delta\}$ has no pair of cut points, then, via TCT, we can 
compare the angle, which is less than $\delta$, 
at $p$ of a geodesic triangle on $M$ with that at $\tilde{p}$ of a comparison one on $\wt{V}(\delta)$, 
see \cite[Theorem 4.12]{KT1}. 
Since, in our case, we need to measure the 
angle $\angle (v_{\bf e}, v_{{\bf e}'})$ between $v_{\bf e}$ 
and $v_{{\bf e}'}$ for any distinct ends ${\bf e}, {\bf e}'\in {\rm Ends}(M)$ and show 
$\B_\lambda (v_{\bf e}) \cap \B_\lambda (v_{{\bf e}'})= \emptyset$, i.e., $\angle (v_{\bf e}, v_{{\bf e}'})\ge 2\lambda$ for some $\lambda$, 
$\wt{V}(\pi)$ must be free from cut points. 
However we do not know 
if $\wt{V}(\pi)$ is free from them. 
We thus employ a noncompact surface of revolution $(M^{*}, p^{*})$ with a metric $g^*:= dt^2 + m(t)^2 d\theta^2$, $(t,\theta) \in (0,\infty) \times \Sph_{p^*}^1$, 
where $m: (0, \infty)\lra \R$ is a positive $C^r$-function satisfying the 
Jacobi equation $m''(t) + \wt{K}_- (t) m(t) = 0$ 
with $m(0) = 0$ and $m'(0) = 1$, 
and we set $\Sph_{p^*}^1:=\{v \in T_{p^*}M^*\,|\,\|v\|=1\}$. 
Note that $r\ge 2$, and that the radial curvature of $M$ at $p$ is bounded from 
below by $\wt{K}_-$, for $\wt{K} \ge \wt{K}_-$ on $[0,\infty)$. 
Moreover, since $g^*$ is at least of class $C^2$ and $\wt{K}_-\le 0$, 
we see that 
\begin{itemize}
\item 
the sector $V^*(\pi):=\{x^*\in M^*\,|\,0<\theta(x^*) < \pi\}$ has no pair of cut points; 
\item the proof of a new type of TCT (\cite[Theorem 4.12]{KT1}) 
works for our case,
\end{itemize}
which are the reason why $(M^{*}, p^{*})$ is employed. 
Hence we have a new type of TCT for $(M, p)$ whose 
reference surface is $(M^{*}, p^{*})$, and can 
apply TCT to all geodesic triangles with $p$ 
as one of their vertices on $M$. 
In the same manner as in the proof of \cite[Theorem C]{KO}, 
we can see that 
\begin{equation}\label{2019_11_09_lem3.3_7}
\angle (v_{\bf e}, v_{{\bf e}'})\ge 2\lambda:=
\frac{2\pi^2}{2\pi -c(M^*)}=\frac{\pi}{\lim_{t\to\infty} m'(t)}
\end{equation}
where $c(M^*)$ denotes the total curvature of $M^*$. Note that 
Corollary \ref{2019_11_25_Cor3.3} gives 
$1\le \lim_{t\to\infty} m'(t) < \infty$, see the second one of Remark 
\ref{2020_04_11_Rem1.4}. The packing lemma and 
Eq.\,\eqref{2019_11_09_lem3.3_7} show
$
\#{\rm Ends}(M) =\# \{ \B_\lambda (v_{\bf e}) \}_{{\bf e} \in {\rm Ends}(M)} \le 2(\pi / 2\lambda)^{n-1} 
= 2 \left\{\lim_{t\to\infty}m'(t) \right\}^{n-1}$, 
which is the desired assertion in this lemma.
$\qedd$
\end{proof}

Lemmas \ref{2019_10_28_lem2.1}--\ref{2019_11_09_lem3.3} complete the proof of Theorem \ref{2019_10_28_maintheorem}.$\qedd$

\begin{flushleft}
K.\,Kondo\\
Department of Mathematics, Faculty of Science\\ Okayama University, Okayama City, Okayama 700-8530, Japan\\
{\small e-mail: 
{\tt keikondo@math.okayama-u.ac.jp}}

\medskip

Y.\,Shinoda\\
Division of Mathematics and Physics\\ 
Graduate School of Natural Science and Technology\\
Okayama University, Okayama City, 
Okayama 700-8530, Japan\\
{\small e-mail: 
{\tt pr648hxt@s.okayama-u.ac.jp}}

\end{flushleft}


\addcontentsline{toc}{section}{References}
\begin{thebibliography}{M}%%%%%%%%%%%%%%%%%%%%%%%%%%%%%%%%%%%%%%%%%%%%%%%%%%%%%%%%%%%%%%%%%%%%%%%%%%%%%%%%%%%%%%%%%%%%%%%%%%%%%%%%%%%%%%%%%%%%%%%%%%%%%%%%%%%%%%%%%%%%%%%%%%%%%%%%%%%%%%%%%%%%%%%%%%%%%%%%
{\small 
\bibitem{A}
U.\,Abresch, 
{\em Lower curvature bounds, 
Toponogov's theorem, and bounded topology}, 
Ann. Sci. \'Ecole Norm. Sup. (4) \textbf{18} (1985), 
no. 4, 651--670. 

\bibitem{AG}
U.\,Abresch and D.\,Gromoll, 
{\em On complete manifolds with nonnegative Ricci curvature}, 
J. Amer. Math. Soc.  
\textbf{3} (1990), No.\,2, 355--374.

\bibitem{BC}
R.L.\,Bishop and R.J.\,Crittenden, 
Geometry of manifolds, 
Pure and Applied Mathematics, Vol. \textbf{XV}, 
Academic Press, New York--London, 1964.

\bibitem{CG2}
J.\,Cheeger and D.\,Gromoll, {\em 
On the structure of complete manifolds of nonnegative curvature}, Ann. of Math. (2) \textbf{96} (1972), 413--443.

\bibitem{CV1}
S.\,Cohn-Vossen, 
{\em 
K\"urzeste Wege und Totalkr\"ummung auf Fl\"achen}, 
Compositio Math. \textbf{2}  (1935), 69--133.

\bibitem{FMS}
P.\,Freitas, J.\,Mao, and I.\,Salavessa, 
{\em Spherical symmetrization and the first eigenvalue of geodesic disks on manifolds}, 
Calc. Var. Partial Differential Equations \textbf{51} (2014), no. 3-4, 701--724. 

\bibitem{G}
M.\,Gromov, 
{\em Curvature, diameter and Betti numbers}, 
Comment. Math. Helv. \textbf{56} (1981), no. 2, 179--195. 

\bibitem{GS}
K.\,Grove and K.\,Shiohama, {\em A generalized sphere theorem},
Ann.\ of Math. (2) \textbf{106} (1977), 201--211.

\bibitem{HI1} J.J.\,Hebda and Y.\,Ikeda, 
{\em Replacing the lower curvature bound in Toponogov's
comparison theorem by a weaker hypothesis}, Tohoku Math. J. (2) \textbf{69} (2017), no.~2, 305--325.

\bibitem{HI2} J.J.\,Hebda and Y.\,Ikeda, 
{\em Necessary and sufficient conditions for a triangle comparison 
theorem}, {\tt arXiv:1806.04633}.

\bibitem{Hub}
A.\,Huber, {\em On subharmonic functions and differential geometry in the large}, 
Comment. Math. Helv. \textbf{32} (1957), 13--72. 

\bibitem{KK}
N.N.\,Katz and K.\,Kondo, 
{\em Generalized space forms}, 
Trans. Amer. Math. Soc. \textbf{354} (2002), 
2279--2284. 

\bibitem{KO}
K.\,Kondo and S.\,Ohta, 
{\em Topology of complete manifolds with radial curvature bounded from below}, 
Geom. Funct. Anal. \textbf{17} (2007), no. 4, 1237--1247. 

\bibitem{KT1}
K.\,Kondo and M.\,Tanaka, 
{\em Total curvatures of model surfaces control topology of complete open manifolds with radial curvature bounded below.\,II}, 
Trans. Amer. Math. Soc. \textbf{362} (2010),  
6293--6324.

\bibitem{KT2}
K.\,Kondo and M.\,Tanaka, 
{\em Total curvatures of model surfaces control topology of complete open manifolds with radial curvature bounded below.\,I}, 
Math. Ann. \textbf{351} (2011), 251--266.　

\bibitem{KT4}
K.\,Kondo and M.\,Tanaka, 
{\em Total curvatures of model surfaces control topology of complete open manifolds with radial curvature bounded below.\,III}, 
J. Math. Soc. Japan \textbf{64} (2012), 
185--200.　

\bibitem{Mao1}
J.\,Mao, {\em Volume comparison theorems for manifolds with radial curvature bounded}, 
Czechoslovak Math. J. \textbf{66 (141)} (2016), no. 1, 71--86. 

\bibitem{Mao2}
J.\,Mao, {\em 
The Gagliardo-Nirenberg inequalities and manifolds with non-negative weighted Ricci curvature}, 
Kyushu J. Math. \textbf{70} (2016), no. 1, 29--46. 

\bibitem{Sh}
K.\,Shiohama, {\em The role of total curvature on complete noncompact Riemannian $2$-manifolds},
Illinois J.\,Math. \textbf{28} (1984), 597--620.

\bibitem{SST}
K.\,Shiohama, T.\,Shioya, and M.\,Tanaka, 
The Geometry of Total Curvature on Complete Open Surfaces, 
Cambridge tracts in mathematics \textbf{159}, 
Cambridge University Press, Cambridge, 2003.

\bibitem{KT5}
M.\,Tanaka and K.\,Kondo, 
{\em The topology of an open manifold with radial curvature bounded from below by a model surface with finite total curvature and examples of model surfaces},  Nagoya Math. J. \textbf{209} (2013), 23--34.
}
\end{thebibliography}
\end{document}